\documentclass[a4paper,11pt]{amsart}

\usepackage{
a4wide,
amsfonts,
amssymb,
amscd,
amsmath,
latexsym,
amsbsy,
enumitem,
todonotes
}

\newtheorem{thm}{Theorem}[section]

\newtheorem{lem}[thm]{Lemma}
\newtheorem{cor}[thm]{Corollary}
\newtheorem{prop}[thm]{Proposition}
\newtheorem{example}[thm]{Example}
\theoremstyle{definition}
\newtheorem{rmk}[thm]{Remark}

\numberwithin{equation}{section}

\def\al{\alpha}

\def\de{\delta}
\def\la{\lambda}

\def\Ga{\Gamma}
\def\De{\Delta}

\def\C{\mathbb{C}}
\def\N{\mathbb{N}}
\def\cD{\mathcal D}

\newcommand{\rFs}[5]{\,_{#1}F_{#2} \left( \genfrac{.}{.}{0pt}{}{#3}{#4};#5 \right)}

\newcommand{\SU}{\mathrm{SU}}
\newcommand{\U}{\mathrm{U}}

\newcommand{\bbC}{\mathbb{C}}

\title[Matrix valued Laguerre polynomials]{Matrix valued Laguerre polynomials}

 \author{Erik Koelink}
 \address{IMAPP, Radboud Universiteit, 
 Heyendaalseweg 135, 
 6525 GL Nijmegen, 
 the Netherlands}
 \email{e.koelink@math.ru.nl}
 \author{Pablo Rom\'an}
 \address{CIEM,
 FaMAF, Universidad Nacional de C\'ordoba, Medina Allende s/n Ciudad
 Universitaria, C\'ordoba, Argentina}
 \email{roman@famaf.unc.edu.ar}
\date{\today}

\begin{document}

\dedicatory{Dedicated to Ben de Pagter on the occasion of his retirement. \\
In admiration of his mathematics, personal leadership and calmness.}

\begin{abstract} Matrix valued Laguerre polynomials are introduced via 
a matrix weight function involving several degrees of freedom using the 
matrix nature. Under suitable conditions on the parameters the 
matrix weight function satisfies matrix Pearson equations, which allow 
to introduce shift operators for these polynomials. 
The shift operators lead to explicit expressions for the structures of
these matrix valued Laguerre polynomials, such as a Rodrigues 
formula, the coefficients in the three-term recurrence, differential
operators, and expansion formulas. 
\end{abstract}

\maketitle

\section{Introduction} 

Matrix valued polynomials have a long history, introduced in the 1940s by M.G.~Krein 
in a study of the matrix moment problem as well as in the study of operators 
with higher deficiency indices \cite{Krein1}, \cite{Krein2}. Since its introduction
several other applications have been linked to matrix valued orthogonal polynomials,
such as relations and applications to e.g. spectral theory 
of higher order recurrences, scattering theory and differential operators, see e.g. \cite{ApteN}, \cite{DuraVA}, \cite{Gero}, \cite{GroeIK}, 
\cite{GroeK}.
As in the study of classical, i.e. scalar-valued, polynomials there are several possible approaches, and 
this can be roughly divided in general methods and studying special classes.
More general methods can be methods such as studying the general moment problem, general 
approximation properties, etc. The special classes are typically introduced as arising in or 
related to the 
study of other topics, such as solutions to differential operators, in terms of 
hypergeometric functions, numerical analyis, representation theory, applications in 
mathematical physics, etc. 
This is particularly true for the matrix valued case of orthogonal polynomials, but the 
theory is not as fully developed as that of orthogonal polynomials. We
refer to Damanik et al. \cite{DamaPS} for an overview of theory and references. 

Important examples of the special cases are the ones that arise in representation 
theory via matrix valued spherical functions. Even though the 
general spherical functions go back at least to the
work of Godement in the 1950s, see references in \cite{GangV}, the relation to 
matrix valued orthogonal polynomials is of a much later time. 
A first step was made by 
Koornwinder \cite{Koor-SIAM85} realizing vector-valued orthogonal polynomials, and 
next Gr\"unbaum, Pacharoni, Tirao \cite{GrunPT} studied matrix valued orthogonal
polynomials related to the symmetric pair $(\SU(3), \U(2))$ using differential
operators. These papers have been motivational to understand the situation 
of matrix valued orthogonal polynomials in connection to representation theory,
and in particular we refer to \cite{Alde}, \cite{HeckvP}, 
\cite{KoelvPR1}, \cite{KoelvPR2}, \cite{KoelvPR3}, \cite{Prui}, \cite{vPR}.

It is well-known that in the scalar-case many polynomials in the Askey scheme,
see \cite{AskeW}, \cite{KoekLS}, \cite{KoekS}, have an interpretation in 
representation theory, but typically only for a very limited set of parameters
involved. It is then an important question how to extend to a general set of 
parameters, and this is a non-trivial question since usually the limited set is finite.
An important example is the extension of multivariable spherical functions on
Riemannian symmetric spaces to the general solutions of integrable systems 
in terms of Heckman-Opdam functions, see Heckman's lectures in \cite{HeckS}. 
For the case of the extension of \cite{KoelvPR1}, \cite{KoelvPR2} from
matrix valued Chebyshev polynomials to matrix valued Gegenbauer polynomials,
this is done in \cite{Koe:Rio:Rom}. 

In \cite{Koe:Rio:Rom} it turns out that one of the essential features that 
make the analytic extension work is the appropriate $LDU$-decomposition 
of the matrix weight. Here $L$ is an unipotent lower triangular matrix 
consisting of Gegenbauer polynomials, and the diagonal matrix $D$ is 
intimately related to the weight of the classical Gegenbauer polynomials.
These features make it possible to introduce matrix valued analogues of 
shift operators, which make it possible to derive properties in an 
explicit fashion, cf. \cite{HeckS} for the importance of shift operators
in the multivariable setting. 
Moreover, it turns out that the inverse matrix polynomial $L$ has a simple 
expression in terms of Gegenbauer polynomials, due to Cagliero and Koornwinder 
\cite{CaglK}. 
It is then a natural question to find the appropriate analogues for the 
simplest classical polynomials in the Askey scheme, i.e. the Hermite and 
Laguerre polynomials. For the Hermite polynomials it is possible
to use a limit from Gegenbauer polynomials, and this suggests how to
deal with the appropriate matrix valued analogue of the Hermite
polynomials, see \cite{IsmaKR-MV}. It tuns out that then the matrix polynomial
$L$ is, up to a constant matrix, a matrix exponential, and it is shown
that, for certain values of the parameters, it is related to examples by 
Gr\"unbaum and Dur\'an in \cite{DuraG}, \cite{DuraG2} and \cite{Dura2}. 
Again, as shown in \cite{IsmaKR-MV}, the shift operators allow us to 
be explicit in describing the properties of these matrix polynomials.

The purpose of this paper is to describe explicitly the matrix valued
analogue of the Laguerre polynomials following the approach sketched above.
So we introduce in Section \ref{sec:MVLaguerrepol} 
the matrix weight function in terms of an $LDU$-decomposition 
where the matrix polynomial $L$ is an unipotent lower triangular matrix
with entries given in terms of Laguerre polynomials in analogy with 
the $LDU$-decomposition of the matrix weight in \cite{Koe:Rio:Rom}.
The corresponding monic matrix valued Laguerre polynomials are eigenfunctions
of a second order matrix valued differential operator which is symmetric
with respect to the weight, see Section \ref{sec:symmetricDE}. 
In Section \ref{sec:Pearson} we use the degrees of freedom introduced 
in the matrix weight to impose two non-linear conditions, which allow us to derive
Pearson equations for the weight. For specific values of the parameters, the 
matrix-valued Laguerre polynomials are closely related to the family studied in \cite{DuradlI}. 
In turn, in Section \ref{sec:shift}
this allows us to calculate the squared norm explicitly, and we 
give an Rodrigues formula to generate the polynomials. 
We naturally obtain from the Pearson equations a matrix valued differential
operator which is factored as lowering and raising operator. We
calculate its Darboux transform in terms of the matrix differential operator
obtained before. As an application of the shift operators, we give
expressions for the matrix coefficients in the three-term recurrence relation 
and we apply it to find an expansion formula for the matrix valued 
Laguerre polynomials. 
We briefly describe three families of solutions to the conditions. 
In particular, this family of matrix valued orthogonal polynomials fits into 
the setting of Cantero, Moral and Vel\'azquez \cite{CantMV}, since the derivative
of the matrix valued Laguerre polynomials form a family of matrix valued 
orthogonal polynomials as well. However, \cite{CantMV} does not contain 
any explicit family of arbitrary size of comparable complexity as in this paper. 
For convenience, Section \ref{sec:prelim} describes some basic results
on matrix valued orthogonal polynomials. 

As mentioned above, the matrix valued analogue of the Laguerre polynomials
in this paper is not motivated by a limit transition of the matrix valued
Gegenbauer polynomials \cite{Koe:Rio:Rom}, which find their origin in 
group theory and the use of shift operators. This is the case for the 
Hermite polynomials of \cite{IsmaKR-MV}, and it turns out that the 
matrix valued Gegenbauer polynomials and matrix valued Hermite polynomials
share some features, such as the simple structure of the matrices in the 
three-term recurrence relation. Indeed, the $B_n$ and $C_n$ of \eqref{eq:three-term-r}
are tridiagonal and diagonal, respectively. This allows, with the results 
of \cite[Lemma~3.1]{KR15}, to give a simple proof of the triviality of 
$\mathcal{A}_W$ and $A_W$, see Section \ref{sec:prelim}. In the case of the 
matrix valued Laguerre polynomials the expressions for $B_n$ and $C_n$
are explicit, but more complex, see Proposition \ref{prop:3termrecur} 
and a proof of the irreducibility along  these lines does not seem to work. 
We conjecture that $\mathcal{A}_W$ and $A_W$ 
are trivial as well, but we have no full proof. 
Another important difference is that in the case of the  matrix valued Hermite, respectively Gegenbauer, 
polynomials, the entries of these polynomials can be explicitly calculated in terms
of the scalar Hermite polynomials and Hahn polynomials, respectively Gegenbauer and Racah polynomials,
see \cite[Thm.~3.13]{IsmaKR-MV}, \cite[Thm.~3.4]{Koe:Rio:Rom}. In the Laguerre setting of this paper,
this is not possible because the eigenvalue matrix of Proposition \ref{prop:symmetry_D} is not 
diagonal.

\medskip
\noindent
\textbf{Acknowledgements.}
Erik Koelink gratefully acknowledges the support and hospitality of FaMAF at Universidad Nacional 
de C\'ordoba and the support of an Erasmus+ travel grant.
The work of Pablo Rom\'an was supported by Radboud Excellence Fellowship, CONICET grant PIP 112-200801-01533, FONCyT grant PICT 2014-3452 and by SeCyT-UNC.

\section{Preliminaries}\label{sec:prelim}

In this section we give some background on matrix valued orthogonal polynomials and we fix the notation, 
for more information we refer to e.g. \cite{DamaPS}, \cite{DuraG}, \cite{GrunT}. 

We consider a complex $N\times N$ matrix valued integrable function 
$W\colon (a,b)\to M_N(\mathbb{C})$, and we 
assume that $W(x)>0$ almost everywhere and such that the weight has finite moments of all orders.
Here $M_N(\mathbb{C})$ is the algebra of complex $N\times N$-matrices.
Integration of a matrix valued function is separate integration of each matrix entry, so that 
the integral is a matrix. 
Here we allow the endpoints $a$ and $b$ to be infinite. 
By $M_N(\mathbb{C})[x]$ we denote the algebra over $\mathbb{C}$  of all polynomials in 
$x$ with coefficients in $M_N(\mathbb{C})$. 
The weight matrix $W$ induces a Hermitian sesquilinear form
\begin{equation}
\label{eq:HermitianForm}
\langle P,Q \rangle=\int_a^b P(x)W(x)Q(x)^*dx \in M_N(\C),
\end{equation}
such that for all $P,Q,R\in M_N(\bbC)[x]$, $T\in M_N(\bbC)$ and $a,b\in \mathbb{C}$ 
the following properties are satisfied
$$
\langle aP+bQ,R\rangle=a\langle P,R\rangle+b\langle Q,R\rangle, \qquad \langle TP,Q\rangle=T\langle P,Q\rangle, \langle P,Q\rangle^*=\langle Q,P\rangle.
$$
Moreover
$$
\langle P,P \rangle = 0 \quad \Longleftrightarrow  \quad P=0,
$$
see for instance \cite{GrunT}. 
Given a weight matrix $W$ there exists a unique sequence of monic orthogonal polynomials $\{P_n\}_{n\geq 0}$ 
in $M_N(\bbC)[x]$. Another sequence of $\{R_n\}_{n\geq 0}$ of orthogonal polynomials in 
$M_N(\bbC)[x]$ is of the form $R_n(x)=A_nP_n(x)$ for some $A_n\in \operatorname{GL}_N(\mathbb{C})$.

Orthogonal polynomials  satisfy a three-term recurrence relation and for monic
orthogonal polynomials $\{P_n\}_{n\geq 0}$ we have
\begin{equation}
\label{eq:three-term-r}
xP_n(x)=P_{n+1}(x)+B_{n}(x)P_n(x)+C_nP_{n-1}(x), \quad n\geq 0,
\end{equation}
where $P_{-1}=0$ and $B_n$, $C_n$ are matrices depending on $n$ and not on $x$.

We say that a matrix weight $W$  is reducible to weights of smaller size if there exists 
a constant matrix $M$ such that 
$MW(x)M^\ast = \mathrm{diag}(W_1(x),\ldots, W_k(x))$ for all $x\in [a,b]$, where $W_1,\ldots,W_k$, $k>1$, 
are weights of size less than  $N$. In such a case, the real vector space
$$
{\mathcal{A}}_{W}= \{ Y \in \mathrm{Mat}_N(\bbC)  \mid YW(x) = W(x)Y^\ast, \quad \text{for all }x\in [a,b]\},
$$
is non-trivial. On the other hand, if the commutant algebra
$$
A_W= \{Y \in\mathrm{Mat}_N(\bbC) \mid YW(x) = W(x)Y,\quad \text{for all }x\in [a,b]\},
$$
is nontrivial, then the weight $W$ is reducible via a unitary matrix $M$. 
The relation between $\mathcal{A}_W$ and $A_W$ is the following; 
if $\mathcal{A}_{W}$ is $\ast$-invariant, the weight $W$ reduces to weights of smaller size if and only if the commutant algebra $A_{W}$ is not trivial, see \cite{KR15}.

\section{Matrix valued Laguerre-type polynomials}\label{sec:MVLaguerrepol}

We introduce a family of matrix valued weight functions by giving an explicit $LDU$-decomposition, 
and in Section \ref{sec:Pearson} we give conditions on the parameters involved. 
The matrix polynomial $L$ in this decomposition is a unipotent lower triangular matrix polynomial, 
whose entries are multiples of Laguerre polynomials independent of the dimension of the matrix. 
This structure is inspired by a family of matrix valued orthogonal polynomials 
related to the the symmetric pair $(G,K)=(\SU(2)\times \SU(2), \SU(2))$ ($K$ diagonally embedded), 
which involves a lower triangular matrix whose entries are Gegenbauer polynomials, see \cite{KoelvPR1}, \cite{KoelvPR2}, \cite{Koe:Rio:Rom}, as well as by the extension to matrix valued $q$-orthogonal polynomials related to the quantum analogue of $(G,K)$, see \cite{AldeKR}.  
       
In the paper the integer $N\geq 1$ is fixed. Let $\mu=(\mu_1,\ldots,\mu_N)$ be a sequence of non-zero coefficients and $\alpha>0$. Then $L_{\mu}^{(\al)}$ is the $N\times N$ unipotent lower triangular matrix defined by
\begin{equation}
\label{eq:matrix-L-Laguerre}
L^{(\alpha)}_{\mu}(x)_{m,n}=\begin{cases} \frac{\mu_{m}}{\mu_{n}} L^{(\alpha+n)}_{m-n}(x), & m\geq n,\\
0& n<m. 
\end{cases}
\end{equation}
The Laguerre polynomials $L^{(\al)}_n(x)$ can be defined by the generating function
\begin{equation}\label{eq:Laguerre-genfun}
(1-t)^{-1-\al} \exp\left(\frac{xt}{t-1}\right) = \sum_{n=0}^\infty L^{(\al)}_n(x)\, t^n,
\end{equation}
see e.g. \cite[18.12.13]{DLMF}, \cite[(1.11.10)]{KoekS}.
Note that the $\mu$ dependence in \eqref{eq:matrix-L-Laguerre} 
is by conjugation with the matrix $S_\mu=\text{diag}(\mu_1,\ldots,\mu_N)$,
so that $L^{(\alpha)}_{\mu}(x)= S_\mu L^{(\al)}(x) S_\mu^{-1}$, where $L^{(\al)}(x)$ is 
\eqref{eq:matrix-L-Laguerre} with $\mu_i=1$ for all $i$. 
It is a direct consequence of the formula $\frac{d}{dx}L^{(\alpha)}_n=-L^{(\alpha+1)}_{n-1}$,
see e.g. \cite[(1.11.6)]{KoekS}, that the matrix $L^{(\alpha)}$ satisfies
\begin{equation}
\label{eq:derivativeL}
\frac{dL^{(\alpha)}}{dx}(x)=L^{(\alpha)}(x)\, A, \qquad A = -\sum_{k=2}^N E_{k,k-1} \quad \Longrightarrow \quad L^{(\alpha)}(x)=L^{(\alpha)}(0)e^{xA}.
\end{equation}
By conjugation we have $A_\mu= S_\mu A S_\mu^{-1}$, 
$\frac{dL^{(\alpha)}_\mu}{dx}(x)=L^{(\alpha)}_\mu(x)\, A_\mu$ and 
$L^{(\alpha)}_\mu(x)=L^{(\alpha)}_\mu(0)e^{xA_\mu}$. Because of the 
dependence of the parameter of the Laguerre polynomial on $n$ in \eqref{eq:matrix-L-Laguerre},
$L^{(\al)}(x)$ does not commute with $A$, see Lemma \ref{lem:commutation_reltatios}. 

\begin{rmk}\label{rmk:HermiteLaguerre}
Comparing the expression $L^{(\al)}(x)$ of \eqref{eq:matrix-L-Laguerre}, \eqref{eq:derivativeL}
with the corresponding $L$ operator in the matrix valued Hermite case, see \cite[(3.1), Prop.~3.1]{IsmaKR-MV},
we see that the exponential part of $L^{(\al)}(x)$ and of the $L$ in \cite{IsmaKR-MV} is the same up to 
a factor $-2$. So, denoting $L$ of \cite{IsmaKR-MV} for the moment by $H$, we see that 
$L^{(\al)}(-2x) = L^{(\al)}(0) H(0)^{-1}H(x)$. Taking the $(m,n)$-entry we get connection coefficients relating
Laguerre polynomials and Hermite polynomials, explicitly
\[
L^{(\al+n)}_{m-n}(-2x) = \sum_{p=n}^m \bigl( L^{(\al)}(0) H(0)^{-1} \bigr)_{m,p} \frac{1}{(p-n)!} H_{p-n}(x)
\]
where $H_n$ are the standard Hermite polynomials \cite[\S 1.13]{KoekS}. Since $H(0)^{-1}$ is known
by \cite[Prop.~3.1]{IsmaKR-MV}, we can rewrite the $\bigl( L^{(\al)}(0) H(0)^{-1} \bigr)_{m,p}$ 
explicitly as a sum, and we obtain
\begin{equation}\label{eq:conncoefHermiteLaguerre}
L^{(\al+n)}_{m-n}(-2x) = \sum_{p=n}^m \frac{(\al+p+1)_{m-p}}{(m-p)!\, (p-n)!}
\rFs{2}{2}{\frac12(p-m), \frac12(p-m+1)}{\frac12(\al+p+1), \frac12(\al+p+2)}{1}
H_{p-n}(x). 
\end{equation}
\end{rmk}

Observe that the inverse of $L^{(\alpha)}_{\mu}$ is a unipotent lower triangular polynomial matrix function. 
Hence, its inverse  is a unipotent lower triangular polynomial matrix function
as well. 
In order to obtain an explicit expression for $(L^{(\alpha)}_{\mu}(x))^{-1}$, we need the following result of \cite[\S 1]{Koekoek1999}, which can be obtained from the generating function 
\eqref{eq:Laguerre-genfun} for the Laguerre polynomials.

\begin{lem}\label{lem:Linverse}
For $i,j\in \mathbb{N}$, $i\geq j$, the following inversion formula for the Laguerre polynomials holds true.
$$
\sum_{k=j}^{i} L_{i-k}^{(-\alpha-i-1)}(-x)L_{k-j}^{(\alpha+j)}(x)=\delta_{i,j}.
$$
\end{lem}

\begin{cor}
The inverse of $L^{(\alpha)}_{\mu}$ is given explicitly by
$$
(L^{(\alpha)}_{\mu}(x)^{-1})_{m,n}=
\begin{cases} \frac{\mu_{m}}{\mu_{n}} \, L_{m-n}^{(-\alpha-m-1)}(-x), & m\geq n,\\
0& n<m.
\end{cases}
$$
\end{cor}

\begin{proof}
For $r\geq s$, we have
\begin{align*}
\left(L^{(\alpha)}(x)^{-1}L^{(\alpha)}(x)\right)_{r,s} = \sum_{k=1}^N (L^{(\alpha)}(x))^{-1}_{r,k} \, 
L^{(\alpha)}(x)_{k,s}= \sum_{k=s}^r  L^{(-\alpha-r-1)}_{r-k}(-x)L^{(\alpha+s)}_{k-s}(x)= \delta_{r,s}.
\end{align*}
The case $r<s$ follows with an analogous calculation, and now conjugate with $S_\mu$. 
\end{proof}

Note that Lemma \ref{lem:Linverse} follows directly from a generating function, and the same holds for 
the Hermite case \cite{IsmaKR-MV}. For the other cases \cite{KoelvPR2}, \cite{Koe:Rio:Rom} this 
is more involved, and the necessary result has been obtained by Cagliero and Koornwinder \cite{CaglK}
and for the Gegenbauer case a $q$-analogue is given by Aldenhoven \cite{Alde}.

Writing $L_\mu^{(\alpha)}(x)$ as the product of a constant matrix times an exponential function as in \eqref{eq:derivativeL} allows us to give an explicit relation between the matrices with different parameters.

\begin{lem}\label{lem:relation-alpha-lambda}
Let $L_\mu^{(\alpha)}(x)$ be the matrix polynomial as in \eqref{eq:matrix-L-Laguerre} and let $\lambda>0$. Then 
$$
L_\mu^{(\alpha)}(x)=M^{(\alpha,\lambda)}_\mu L_\mu^{(\lambda)}(x),\qquad 
M^{(\alpha,\lambda)}_\mu=L_\mu^{(\alpha)}(x)L_\mu^{(\lambda)}(x)^{-1}=
L_\mu^{(\alpha)}(0)L_\mu^{(\lambda)}(0)^{-1}.
$$
Moreover,  $M^{(\alpha,\lambda)}_\mu$ is explicitly given by
$$
M^{(\alpha,\lambda)}_\mu=\sum_{k=0}^{N-1} \frac{(\alpha-\lambda)_k}{k!} (-1)^k A_\mu^k.
$$
\end{lem}

Note that the series for $M^{(\alpha,\lambda)}_\mu$ terminates, since $A_\mu$ is nilpotent.
Using the binomial theorem it can also be expressed as 
\[
M^{(\alpha,\lambda)}_\mu  =(1+A_\mu)^{\la-\al}. 
\]

\begin{proof} 
Note that it suffices to deal with the case $\mu_i=1$ for all $i$, and next conjugate by $S_\mu$.
Since $L^{(\alpha)}(x)$ and $(L^{(\lambda)}(x))^{-1}$ are unipotent lower triangular matrices, 
$M^{(\alpha,\lambda)}$ is unipotent lower triangular. Moreover, $M^{(\alpha,\lambda)}$ 
is constant, since the exponentials cancel.
Now the $(r,s)$-entry of $M^{(\alpha,\lambda)}L^{(\lambda)}(x)=L^{(\alpha)}(x)$, 
for $r>s$, is
\begin{align*}
\sum_{k=s}^r M^{(\alpha,\lambda)}_{r,k} L^{(\lambda+s)}_{k-s}(x)=  L^{(\alpha+s)}_{r-s}= 
\sum_{k=s}^{r} \frac{(\alpha-\lambda)_{r-k}}{(r-k)!} \, L^{(\lambda+s)}_{k-s},
\end{align*}
where we use \cite[18.18.18]{DLMF} in order to write the Laguerre polynomials with parameter 
$\lambda +s$ in terms of Laguerre polynomials with paramater $\alpha+s$. 
Note that the required identity \cite[18.18.18]{DLMF} is a direct consequence of the 
same generating function \eqref{eq:Laguerre-genfun} and the binomial theorem. 
We obtain
$$M^{(\alpha,\lambda)}_{r,k} = \frac{(\alpha-\lambda)_{r-k}}{(r-k)!}.$$
Conjugating by $S_\mu$, the lemma follows.
\end{proof}

We are now ready to introduce the weight matrix. For $\nu>0$ we define the matrix
\begin{equation}
\label{eq:weight_factorization}
W^{(\alpha,\nu)}_\mu(x)=L_\mu^{(\alpha)}(x)\,T^{(\nu)}(x)\,L_\mu^{(\alpha)}(x)^\ast,\qquad T^{(\nu)}(x)=e^{-x} \sum_{k=1}^N x^{\nu+k} \delta_k^{(\nu)}\, E_{k,k}.
\end{equation}
Here the parameters $\delta_k^{(\nu)}$, $1\leq k\leq N$, are to be determined later, see conditions 
\eqref{eq:condition_delta_Phi} and \eqref{eq:recursion-alphas} in Section \ref{sec:Pearson}.
For now we assume the condition $\delta_k^{(\nu)}>0$, $1\leq k\leq N$, so that the weight is positive definite. 
Moreover, since $\nu>0$ the factor $e^{-x}x^{\nu+k}$ in the entries of the diagonal matrix $T^{(\nu)}$ guarantees that all the moments exist.  Note that even $\nu>-2$ suffices for the existence of moments as well, but
we require $\nu>0$ in all cases in Examples \ref{ex:ax1}, \ref{ex:ax2}, \ref{ex:ax3}. 
By $\{P_n^{(\alpha, \nu)}\}_n$ we denote the sequence of monic orthogonal polynomials with respect to $W^{(\alpha, \nu)}_\mu$. We suppress the $\mu$-dependence in the notation for the polynomials and the related 
quantities, such as the squared norms, coefficients in the three-term recurrence relation, etc. 
Note that the structure in $T^{(\nu)}$ is motivated by the results of \cite{KoelvPR2}, \cite{Koe:Rio:Rom}. 

Observe that for real $\mu_{i}$'s we have 
\begin{align}
\label{eq:decompW_1}
W_\mu^{(\alpha, \nu)}(x)&=L_{\mu}^{(\alpha)}(x) T^{(\nu)}(x)   L_{\mu}^{(\alpha)}(x)^\ast = L^{(\alpha)}_\mu(0) e^{xA_\mu} T^{(\nu)}(x) e^{xA_\mu^\ast}  L^{(\alpha)}_\mu(0)^\ast.
\end{align}
From now on we assume that the coefficients $\mu_{i}$ are real and non-zero for all $i$.
 
 Using Lemma \ref{lem:relation-alpha-lambda} we obtain 
 \begin{equation}
 \label{eq:weight-alpha-lambda}
 W^{(\alpha, \nu)}_\mu = M^{(\alpha,\lambda)}_\mu \, W^{(\lambda, \nu)}_\mu (M^{(\alpha,\lambda)}_\mu)^\ast,
 \end{equation}
 for $\alpha, \lambda>0$. Denote by $ H^{(\alpha,\nu)}_n$ the $n$-th squared norms for the monic orthogonal polynomials:
 $$
 H^{(\alpha,\nu)}_n=\int_0^\infty P_n^{(\alpha,\nu)}(x) W^{(\alpha, \nu)}_\mu(x) P_n^{(\alpha,\nu)}(x)^\ast dx,
 $$
so that by \eqref{eq:weight-alpha-lambda} we have the following relation for the squared norms with different parameters
\begin{equation}
\label{eq:norms-diff-parameters}
H^{(\alpha,\nu)}_n=M^{(\alpha,\lambda)}_\mu H^{(\lambda,\nu)}_n (M^{(\alpha,\lambda)}_\mu)^\ast.
\end{equation}
Recall that the squared norms satisfy $H^{(\alpha,\nu)}_n>0$. 

\begin{prop}
\label{prop:norm-zero}
If we let $\alpha=\nu$, then the $0$-th moment $H_0^{(\nu, \nu)}$ is the diagonal matrix
$$(H_0^{(\nu, \nu)})_{j,j}=\int_0^\infty (W_\mu^{(\nu)}(x))_{j,j}\, dx=\frac{\mu_j^2 \, \Gamma(\nu+j+1)}{(j-1)!} \,\sum_{k=1}^j \frac{\delta^{(\nu)}_k}{\mu_k^2} (-1)^{k+1}(-j+1)_{k-1}.$$
\end{prop}
\begin{proof}
It follows from the definition of the weight \eqref{eq:weight_factorization} and the orthogonality relations for the Laguerre polynomials that
\begin{align*}
(H_0^{(\nu, \nu)})_{i,j}=\int_0^\infty (W_\mu^{(\nu, \nu)}(x))_{i,j}\, dx &= \sum_{k=1}^{\mathrm{min}(i,j)} \, \frac{\mu_i\mu_j}{\mu_k^2} \delta^{(\nu)}_k \int_0^\infty L^{(\nu+k)}_{i-k}(x) L^{(\nu+k)}_{j-k}(x) x^{\nu+k}e^{-x} \, dx\\
&= \delta_{i,j} \mu_j^2 \Gamma(\nu+j+1) \sum_{k=1}^N \frac{ \delta^{(\nu)}_k}{\mu_k^2 (j-k)!}.
\end{align*}
Now the proposition follows by rewriting the factor $(j-k)!$. 
\end{proof}

\section{A symmetric second order differential operator}\label{sec:symmetricDE}

A standard technique in order to deal with matrix valued polynomials is to obtain a 
matrix valued differential operator having the matrix valued orthogonal polynomials
as eigenfunctions, see e.g. \cite{Dura-CA}, \cite{DuraG}, \cite{GrundlIMF}, \cite{IsmaKR-MV},
\cite{Koe:Rio:Rom}. 
We obtain a second-order matrix valued differential operator which is symmetric with 
respect to $W^{(\nu)}_\mu$ and which preserves polynomials and its degree, by establishing a conjugation to 
a diagonal matrix differential operator using the approach of \cite{Koe:Rio:Rom}.

Let $F_2$, $F_1$, $F_0$ be matrix valued polynomials of degrees two, one and zero respectively. 
Assume that we have a matrix valued second-order differential operator $D$ which acts on a 
matrix valued $C^{2}([0,\infty))$-function $Q$ by
\begin{equation}
\label{eq:form-differential-operator-general}
QD=\left(\frac{d^2Q}{dx^2}\right)(x) \, F_2(x) + \left(\frac{dQ}{dx}\right)(x) \, F_1(x) +Q(x) F_0(x).
\end{equation}
For a positive definite matrix valued weight $W$ with finite moments of all orders, we say that $D$ is symmetric with respect to $W$ if for all matrix valued $C^{2}([0,\infty))$-functions $G,H$ we have
$$\int_0^\infty (GD)(x)W(x)(H(x))^\ast\, dx = \int_0^\infty G(x)W(x)((HD)(x))^\ast dx.$$
By \cite[Thm~3.1]{DuraG},  $D$ is symmetric with respect to $W$ if and only if
the boundary conditions 
\begin{gather}
\label{eq:symmetry-boundary1}
\lim_{x\to a} F_2(x)W(x) = 0 = \lim_{x\to b} F_2(x)W(x), \\ 
\label{eq:symmetry-boundary2}
\lim_{x\to a} F_1(x)W(x) - \frac{d(F_2W)}{dx}(x) = 0 = 
\lim_{x\to b} F_1(x)W(x) - \frac{d(F_2W)}{dx}(x)
\end{gather}
and the symmetry conditions 
\begin{gather}
\label{eq:symmetry-conditions}
F_2(x)W(x) = W(x) \bigl( F_2(x)\bigr)^\ast, \qquad 
2 \frac{d(F_2W)}{dx}(x) - F_1(x)W(x) = W(x) \bigl( F_1(x)\bigr)^\ast, \\
\label{eq:symmetry-conditions2}
\frac{d^2(F_2W)}{dx^2}(x) - \frac{d(F_1W)}{dx}(x) + F_0(x) W(x) = W(x) \bigl( F_0(x)\bigr)^\ast
\end{gather}
for almost all $x\in (a,b)$  hold.

\begin{rmk}
\label{rmk:Dtilde}
Suppose that the differential operator $D$ is symmetric with respect to a weight matrix of the form $W^{(\nu)}_\mu(x)=L(x)\,T(x)\,L(x)^\ast$, and let $\widetilde{D} = \frac{d^2}{dx^2} \widetilde{F}_2 + \frac{d}{dx} \widetilde{F}_1 + \widetilde{F}_0$
be the second-order differential operator obtained by conjugation of $D$ by $L$. Then for all $C^2$-matrix valued functions $Q$ we have
\begin{equation}\label{eq:conjugatedD}
\frac{d^2(QL)}{dx^2}(x) \widetilde{F}_2(x) + \frac{d(QL)}{dx}(x) \widetilde{F}_1(x) + 
(QL)(x)\widetilde{F}_0(x) = 
\bigl( QD\bigr)(x) L(x),
\end{equation}
where the coefficients $F_{i}$ and $\widetilde F_{i}$ are related by
\begin{equation}\label{eq:relationsFtilde}
F_2L=L\widetilde{F}_2, \quad
F_1L= 2\frac{dL}{dx} \widetilde{F}_2 + L\widetilde{F}_1, \quad
F_0L= \frac{d^2L}{dx^2} \widetilde{F}_2 + \frac{dL}{dx} \widetilde{F}_1 + L\widetilde{F}_0,
\end{equation}
see the discussion in \cite[\S 4]{Koe:Rio:Rom}. Moreover, it follows from \cite[Proposition 4.2]{Koe:Rio:Rom} that $D$ is symmetric with respect to $W$ if and only if $\widetilde D$ is symmetric with respect to $T$.
\end{rmk}

In order to obtain the explicit expression for a symmetric differential operator having the matrix valued orthogonal polynomials $P_n^{(\nu)}$ as eigenfunctions, we need to control commutation
relations with some explicit matrices. In particular, $J$ is the diagonal matrix $J_{k,k}=k$. 
Note that $JS_\mu=S_\mu J$. 

\begin{lem}
\label{lem:commutation_reltatios}
The following commutation relations for $L^{(\alpha)}_{\mu}$, $A_{\mu}$ and $J$ hold.
\begin{align*}
L^{(\alpha)}_{\mu}(x)J(L^{(\alpha)}_{\mu}(x))^{-1}&=x(A_{\mu}+1)^{-1}-x+(\alpha+J)A_{\mu}+J,\\
(L^{(\alpha)}_{\mu}(x))^{-1}JL^{(\alpha)}_{\mu}(x)&=J-(\alpha+J-x)A_{\mu}-xA_{\mu}^2,\\
L^{(\alpha)}_{\mu}(x)(1-A_{\mu})(L^{(\alpha)}_{\mu}(x))^{-1}&=(1+A_{\mu})^{-1}.
\end{align*}
\end{lem}

Note that in particular, the last equality gives $[A_\mu, L^{(\al)}_\mu(x)] =A_\mu L^{(\al)}_\mu(x) A_\mu$.

\begin{proof}
It suffices to prove the lemma for $\mu_i=1$ for all $i$. 
We first prove the last identity. If we multiply the third equation by $L^{(\alpha)}(x)$ 
on the right and by $A+1$ on the left, we see that it suffices to show 
$(A+1)L^{(\alpha)}(x)A=AL^{(\alpha)}(x)$. 
The $(r,s)$-entry of this equation is given by
$$
L^{(\alpha+s+1)}_{r-s-2}(x)-L^{(\alpha+s+1)}_{r-s-1}(x)
=-L^{(\alpha+s)}_{r-s-1}(x),
$$
which is e.g. \cite[18.9.14]{DLMF}. 

Similarly, the first equation 
is equivalent to 
$$
(A+1)L^{(\alpha)}(x)J=((\alpha+J-x)A+AJ+J+\al A^2+AJA)L^{(\alpha)}(x).
$$
The $(r,s)$-entry of this equation, after simplifying and regrouping terms, gives
$$
-sL^{(\alpha+s)}_{r-s-1}(x)+L^{(\alpha+s)}_{r-s}(x)=
rL^{(\alpha+s)}_{r-s}(x)+(x-\alpha-2r+1)L^{(\alpha+s)}_{r-s-1}(x)+(\alpha+r-1)L^{(\alpha+s)}_{r-s-2}(x).
$$
Observe that all the Laguerre polynomials have the same parameter and this 
follows from the three-term recurrence relation of the Laguerre polynomials, see e.g. \cite[(1.11.3)]{KoekS}. 

The second equation follows from the other two. Use the third equation in the first one twice to 
rewrite the first equation as
$$
L^{(\al)}(x)J(L^{(\al)}(x))^{-1} = 
- x L^{(\al)}(x)A(L^{(\al)}(x))^{-1} +\al A + J L^{(\al)}(x)(1-A)^{-1}(L^{(\al)}(x))^{-1}
$$
and isolating $J$ from the last term gives 
$$
J = L^{(\al)}(x)\bigl( J(1-A)+xA(1-A) -\al (L^{(\al)}(x))^{-1} A L^{(\al)}(x) (1-A)
\bigr)(L^{(\al)}(x))^{-1}
$$
and use the last equation once more to see $(L^{(\al)}(x))^{-1} A L^{(\al)}(x)(1-A)=A$. 
Rewriting gives the second equation. 
\end{proof}

Now we are ready to write explicitly a symmetric second order differential operator.

\begin{prop}
\label{prop:symmetry_D}
Let $D^{(\alpha, \nu)}$ be given by
$$D^{(\alpha, \nu)}= \left(\frac{d^2}{dx^2}\right) \, F^{(\alpha, \nu)}_{2}(x) +\left(\frac{d}{dx}\right) \, F^{(\alpha, \nu)}_{1}(x) +F^{(\alpha, \nu)}_{0}(x),$$
with $F^{(\alpha, \nu)}_{2}(x)=x$ and
$$F_{1}^{(\alpha, \nu)}(x)=-x(A_\mu+1)^{-1}+\nu+J+1+(\alpha+J)A_\mu, \qquad F^{(\alpha, \nu)}_{0}(x)=(\alpha-\nu)(A_{\mu}+1)^{-1}-J.$$
Then $D^{(\alpha, \nu)}$ is symmetric with respect to $W^{(\alpha, \nu)}_{\mu}$. Moreover,
$$P_{n}^{(\alpha, \nu)}D^{(\alpha, \nu)}=\Gamma^{(\alpha, \nu)}_nP_{n}^{(\alpha, \nu)}, \qquad \Gamma^{(\alpha, \nu)}_n=(-n+\alpha-\nu)(A_{\mu}+1)^{-1}-J, \qquad n\in \N.$$
\end{prop}

Note that the eigenvalue matrix $\Ga^{(\al,\nu)}_n$ is a lower triangular matrix, which also depends 
on the choice of the sequence $\mu$. 

\begin{proof}
Let us consider the differential operator $\widetilde{D}^{(\alpha, \nu)} = \frac{d^2}{dx^2} \widetilde{F}^{(\alpha, \nu)}_2 + \frac{d}{dx} \widetilde{F}^{(\alpha, \nu)}_1 + \widetilde{F}^{(\alpha, \nu)}_0$ obtained by conjugation of $D^{(\alpha, \nu)}$ by the matrix $L^{(\alpha)}_{\mu}$.  Then it follows from \eqref{eq:relationsFtilde} that the coefficients $\widetilde{F}^{(\alpha, \nu)}_i$ are given by $\widetilde{F}^{(\alpha, \nu)}_2(x)=x$ and
\begin{equation}
\label{eq:coefficients-differential-Hermite-conjugated}
\begin{split}
\widetilde F^{(\alpha, \nu)}_{1}&=(L^{(\alpha)}_{\mu})^{-1} \left(F^{(\nu)}_1L^{(\alpha)}_{\mu} -2x\frac{dL^{(\alpha)}_{\mu}}{dx} \right), \\
\widetilde F^{(\alpha, \nu)}_0 &=(L^{(\alpha)}_{\mu})^{-1} \left(F^{(\alpha, \nu)}_0L^{(\alpha)}_{\mu} - x \frac{d^2L^{(\alpha)}_{\mu}}{dx^2} - \frac{dL^{(\alpha)}_{\mu}}{dx} \widetilde{F}^{(\alpha, \nu)}_1\right).
\end{split}
\end{equation}
The derivatives in \eqref{eq:coefficients-differential-Hermite-conjugated} can be 
evaluated by \eqref{eq:derivativeL}. 
It follows from the definition of $F_1^{(\al,\nu)}$ that
\begin{multline*}
\widetilde F^{(\alpha, \nu)}_{1}(x)=\nu+1-x L^{(\alpha)}_{\mu}(x)^{-1}(A_\mu+1)^{-1}L^{(\alpha)}_{\mu}(x)+L^{(\alpha)}_{\mu}(x)^{-1}JL^{(\alpha)}_{\mu}(x)\\
+L^{(\alpha)}_{\mu}(x)^{-1}(\alpha+J)A_\mu L^{(\alpha)}_{\mu}(x)-2xA_\mu.
\end{multline*}
We use the first equation of Lemma \ref{lem:commutation_reltatios} to rewrite 
the term $L^{(\alpha)}_{\mu}(x)^{-1}(\alpha+J)A_\mu L^{(\alpha)}_{\mu}(x)$. 
Similarly the last equation of Lemma \ref{lem:commutation_reltatios} is used, and 
we obtain $\widetilde F^{(\al,\nu)}_{1}(x)=\nu+J+1-x$.

Similarly, using \eqref{eq:derivativeL} and  Lemma 
\ref{lem:commutation_reltatios} and the expression for $\widetilde F^{(\al,\nu)}_{1}$ we obtain
\begin{multline*}
\widetilde F^{(\al,\nu)}_{0}(x)=(\alpha-\nu)(1-A_{\mu})-L^{(\alpha)}_{\mu}(x)^{-1}J L^{(\alpha)}_{\mu}(x)
-xA^2_{\mu}-A_\mu(\nu+J+1-x)=(\alpha-\nu)-J.
\end{multline*}

By Remark \ref{rmk:Dtilde}, in order to prove that $D^{(\alpha, \nu)}$ is symmetric with respect to  $W^{(\alpha, \nu)}_{\mu}$ it is enough to prove that $\widetilde D^{(\alpha, \nu)}$ is symmetric with respect to $T^{(\nu)}$, i.e. we need to show that the boundary conditions \eqref{eq:symmetry-boundary1}, \eqref{eq:symmetry-boundary2} and the symmetry conditions \eqref{eq:symmetry-conditions}, \eqref{eq:symmetry-conditions2} hold true 
with $W$ replaced by $T^{(\nu)}$ and the $F$'s replaced by the 
corresponding $\widetilde F^{(\al,\nu)}$'s. 
Since $\widetilde F^{(\al,\nu)}$'s are polynomials, the weight involves the exponential $e^{-x}$ we see that 
for $\nu>-1$, so in particular for $\nu>0$, the boundary conditions 
\eqref{eq:symmetry-boundary1}, \eqref{eq:symmetry-boundary2} are satisfied. 

The symmetry equation \eqref{eq:symmetry-conditions} and \eqref{eq:symmetry-conditions2} are 
diagonal conditions, and can be verified by a simple calculation.

Hence $D^{(\alpha, \nu)}$ is symmetric with respect to the weigh matrix $W^{(\al,\nu)}_\mu$. 
Since $D^{(\alpha, \nu)}$ preserves polynomials and the degree of the polynomials, $P_{n}^{(\alpha, \nu)}D^{(\alpha, \nu)}$ are 
also orthogonal with respect to $W_\mu^{(\alpha, \nu)}$, so that $P_n^{(\alpha, \nu)}D^{(\alpha, \nu)} = \Gamma^{(\alpha, \nu)}_n P_{n}^{(\alpha, \nu)}$ for 
some matrix $\Gamma^{(\alpha, \nu)}_n$, which is obtained by considering the leading coefficient. 
\end{proof}

\begin{rmk}
It follows from the proof Proposition \eqref{prop:symmetry_D}, that matrix valued polynomials $P_{n}^{(\alpha, \nu)}L^{(\alpha)}_{\mu}$ are polynomial eigenfunctions of the diagonal second-order differential operator $\widetilde D^{(\alpha, \nu)}$ with eigenvalue $\Gamma^{(\alpha, \nu)}_n=-n(A_\mu+1)^{-1}+(\alpha-\nu)(A_{\mu}+1)^{-1}-J$. More precisely, we have
$$ x \frac{d^2(P_{n}^{(\alpha, \nu)}L^{(\alpha)}_{\mu})}{dx^2}(x)  + \frac{d(P_{n}^{(\alpha, \nu)}L^{(\alpha)}_{\mu})}{dx}(x) (\nu+J+1-x) - 
(P_{n}^{(\alpha, \nu)}L^{(\alpha)}_{\mu})(x) J= \Gamma^{(\alpha, \nu)}_n (P_{n}^{(\alpha, \nu)}L^{(\alpha)}_{\mu})(x).$$
Observe that although the differential operator $\widetilde D^{(\alpha, \nu)}$ is diagonal, the eigenvalue $\Gamma^{(\alpha, \nu)}_n$ is a full lower triangular matrix so that the previous equation gives a coupled system of differential equations for the entries of $P_{n}^{(\alpha, \nu)}L_{\mu}^{(\alpha)}$. 
This is in contrast with the  case of matrix valued Gegenbauer polynomials studied in 
\cite{Koe:Rio:Rom}.  An analogous result for a differential operator and involving a diagonal eigenvalue
allows to determine the entries of the analogue of $P_{n}^{(\alpha, \nu)}L_{\mu}^{(\alpha)}$ 
as a single 
Gegenbauer polynomial. This allowed to find explicit expressions for the polynomials, see \cite[\S 5.2]{Koe:Rio:Rom}. The situation of the Gegenbauer setting of \cite{Koe:Rio:Rom} is repeated for the 
matrix valued Hermite polynomials \cite[\S 3]{IsmaKR-MV}. 
\end{rmk}

\section{The matrix valued Pearson equation}\label{sec:Pearson}

In order to establish the existence of shift operators, the Pearson equations are essential. 
We derive the matrix valued Pearson equations for the family of weights $W_{\mu}^{(\alpha, \nu)}$ 
under explicit non-linear conditions relating the coefficients of the sequence $\mu$ and 
the coefficients in $T^{(\nu)}$. 

First we need certain relations involving the function $e^{xA_\mu}$. Recall the diagonal matrix $J$; $J_{k,l}=\de_{k,l}k$. 
Then 
$[J,A_\mu]=A_\mu$ and $[J,A^\ast_\mu]=-A^\ast_\mu$, so that
\begin{equation}
\label{eq:conjAHofJ}
e^{-xA_\mu} J e^{xA_\mu}= xA_\mu+J, \qquad e^{-xA^\ast_\mu} J e^{xA^\ast_\mu}= -xA^\ast_{\mu}+J,
\end{equation}
For this we use that the left hand side of the first expression is a matrix valued polynomial in $x$, 
since $A_\mu$ is nilpotent. Its derivative $e^{-xA_\mu} [J,A_\mu] e^{xA_\mu}=A_\mu$ is constant, and the first formula follows. For the second equation of \eqref{eq:conjAHofJ} we take adjoints and replace $x$ by $-x$ in the first formula.

Now we need to impose conditions on the sequence $\{\mu_i\}_i$ and the coefficients $\delta_{k}^{(\nu)}$. We consider the diagonal matrix $\Delta^{(\nu)} = \mathrm{diag}(\delta_{1}^{(\nu)}, \ldots, \delta_{N}^{(\nu)})$, so that $(T^{(\nu)})_{k,k}=e^{-x} x^{\nu+k} (\Delta^{(\nu)})_{k,k}$. We assume that there exist coefficients $c^{(\nu)}$ and $d^{(\nu)}$ such that $\delta^{(\nu+1)}_{k}=(k\, d^{(\nu)} +c^{(\nu)}) \, \delta^{(\nu)}_{k}$ for all $k=1,\ldots,N$. In other words, we assume that
\begin{equation}
\label{eq:condition_delta_Phi}
\Delta^{(\nu+1)}=(d^{(\nu)} J+c^{(\nu)}) \, \Delta^{(\nu)}.
\end{equation}
Note that $d^{(\nu)}, c^{(\nu)}\geq 0$, since $\delta^{(\nu)}_{k}>0$. In view of \eqref{eq:conjAHofJ}, the condition \eqref{eq:condition_delta_Phi} implies that
\begin{equation}
\label{eq:exponentials_Deltas}
e^{-xA_\mu^\ast} (\Delta^{(\nu)})^{-1} \Delta^{(\nu+1)} e^{xA_\mu^\ast} =e^{-xA_\mu^\ast} (d^{(\nu)} J+c^{(\nu)}) e^{xA_\mu^\ast} = d^{(\nu)}(-x A_\mu^\ast + J)+c^{(\nu)}.
\end{equation}
This is the main ingredient in Proposition \ref{prop:Phinu}.

\begin{prop}
\label{prop:Phinu}
Let $W_\mu^{(\alpha, \nu)}$ be the weight matrix given in \eqref{eq:weight_factorization} and assume that \eqref{eq:condition_delta_Phi} holds true. Then 
$$\Phi^{(\alpha, \nu)}(x)=(W_\mu^{(\alpha, \nu)}(x))^{-1}W_\mu^{(\alpha, \nu+1)}(x),$$
is a matrix valued polynomial of degree two.
\end{prop}

\begin{proof}
Using \eqref{eq:decompW_1}, the fact that 
$T^{(\nu)}(x)^{-1}T^{(\nu+1)}(x) = x(\Delta^{(\nu)})^{-1}\Delta^{(\nu+1)}$ and \eqref{eq:exponentials_Deltas}, we obtain
\begin{align*}
(W^{(\alpha, \nu)}_\mu(x))^{-1}W_\mu^{(\alpha, \nu+1)}(x) 
&= x  (L^{(\alpha)}_\mu(x)^\ast)^{-1} (\Delta^{(\nu)})^{-1} \Delta^{(\nu+1)} L^{(\alpha)}_\mu(x)^\ast \\
& = x(L_\mu^{(\alpha)}(0)^\ast)^{-1} e^{-xA_\mu^\ast} (d^{(\nu)} J+c^{(\nu)}) e^{xA_\mu^\ast} L_\mu^{(\alpha)}(0)^\ast \\
&=-d^{(\nu)}x^2 (L^{(\alpha)}_\mu(0)^\ast)^{-1} A_\mu^\ast L^{(\alpha)}_\mu(0)^\ast + \\
&\qquad\qquad\qquad + x \left(d^{(\nu)} (L^{(\alpha)}_\mu(0)^\ast)^{-1} J L^{(\alpha)}_\mu(0)^\ast + c^{(\nu)}\right). \qedhere
\end{align*}
\end{proof}

Now we assume that the coefficients $\mu_k$ and $\delta^{(\nu)}_k$ satisfy the  relation
\begin{equation}
\label{eq:recursion-alphas}
\frac{\mu_{k+1}^2}{\mu_{k}^2}=d^{(\nu)}k(N-k) \frac{ \delta_{k+1}^{(\nu)} }{\delta_{k}^{(\nu+1)}}, \qquad k=1,\ldots,N-1.
\end{equation}
Note that, since the coefficients $\mu_k$ are independent of $\nu$, we require the right hand side of \eqref{eq:recursion-alphas} to be independent of $\nu$. 

\begin{prop}
\label{prop:Psinu}
Let $W^{(\alpha, \nu)}_\mu$ be the weight matrix given in \eqref{eq:weight_factorization} and assume that the conditions \eqref{eq:condition_delta_Phi} and \eqref{eq:recursion-alphas} hold true. Then
$$\Psi^{(\alpha, \nu)}(x)=(W_\mu^{(\alpha, \nu)}(x))^{-1}\frac{dW_\mu^{(\alpha, \nu+1)}}{dx}(x),$$
is a matrix valued polynomial of degree one.
\end{prop}

\begin{proof}
Using \eqref{eq:decompW_1} we obtain 
\begin{equation*}
\begin{split}
L^{(\alpha)}_\mu(0)^\ast  \Psi^{(\alpha, \nu)}(x) (L^{(\alpha)}_\mu(0)^\ast)^{-1}  = &e^{-xA_\mu^\ast}  (T^{(\nu)}(x))^{-1} A_\mu T^{(\nu+1)}(x)  e^{xA_\mu^\ast} \\
+e^{-xA_\mu^\ast}  (T^{(\nu)})^{-1}  &\frac{dT^{(\nu+1)}}{dx}(x)  e^{xA_\mu^\ast}
 + e^{-xA_\mu^\ast}  (T^{(\nu)})^{-1}  T^{(\nu+1)}  e^{xA_\mu^\ast} A_\mu^\ast.
\end{split}
\end{equation*}
It follows from  \eqref{eq:condition_delta_Phi} and \eqref{eq:conjAHofJ} that
\begin{equation}\label{eq:third-term}
\begin{split}
e^{-xA_\mu^\ast}  T^{(\nu)}(x)^{-1}T^{(\nu+1)}(x)  e^{xA_\mu^\ast}  A_\mu^\ast &=  x e^{-xA_\mu^\ast}  (\Delta^{(\nu)})^{-1} \Delta^{(\nu+1)} e^{xA_\mu^\ast} A_\mu^\ast  \\
&= -x^2 d^{(\nu)} (A_\mu^\ast)^2 + x d^{(\nu)} JA_\mu^\ast + x c^{(\nu)}A_\mu^\ast,
\end{split}
\end{equation}
and
\begin{equation}\label{eq:second-term}
\begin{split}
e^{-xA_\mu^\ast}  &(T^{(\nu)})^{-1}  \frac{dT^{(\nu+1)}}{dx}(x)  e^{xA_\mu^\ast} = x^2 d^{(\nu)} (A_\mu^\ast+(A_\mu^\ast)^2) \\
&- x((A_\mu^\ast+1)(d^{(\nu)}J+c^{(\nu)})+d^{(\nu)}(\nu+J+1)A_\mu^\ast)+(\nu+J+1)(d^{(\nu)}J+c^{(\nu)}).
\end{split}
\end{equation}
We observe that the term $x^2 d^{(\nu)} (A_\mu^\ast)^2$ of the right hand side of \eqref{eq:second-term} cancels with the term of degree two in \eqref{eq:third-term}.

Now we note that  $(T^{(\nu)}(x))^{-1} \, A_\mu \, T^{(\nu+1)}(x) =x  (\Delta^{(\nu)})^{-1} \, A_\mu \,  \Delta^{(\nu+1)}$. On the other hand, the matrix $[(\Delta^{(\nu)})^{-1}  \, A_\mu \,  \Delta^{(\nu+1)}, A_\mu^\ast]$ is a diagonal matrix whose $k$-th diagonal entry is given by
$$
[(\Delta^{(\nu)})^{-1}  \, A_\mu \, \Delta^{(\nu+1)}, A^\ast]_{k,k}=\frac{\mu_{k+1}^2 \delta_{k}^{(\nu+1)}}{\mu_{k}^2\delta^{(\nu)}_{k+1}} -\frac{\mu_k^2 \delta_{k-1}^{(\nu+1)}}{\mu_{k-1}^2\delta^{(\nu)}_k}.
$$
By \eqref{eq:recursion-alphas} we verify that
$[(\Delta^{(\nu)})^{-1} \, A_\mu \, \Delta^{(\nu+1)}, A_\mu^\ast] = 2d^{(\nu)} J-d^{(\nu)}(N+1)$.
This leads to 
\begin{align*}
& \frac{d}{dx} \left( e^{-xA_\mu^\ast}  (\Delta^{(\nu)})^{-1}  \, A_\mu \,  \Delta^{(\nu+1)} e^{xA_\mu^\ast} \right) = e^{-xA_\mu^\ast}  \left[ (\Delta^{(\nu)})^{-1}  \, A_\mu \,  \Delta^{(\nu+1)},A_\mu^\ast \right] e^{xA_\mu^\ast} \\
& \hspace{3cm} = e^{-xA_\mu^\ast} (2d^{(\nu)} J-d^{(\nu)}(N+1))e^{xA_\mu^\ast} \\
& \hspace{3cm} = -2xd^{(\nu)}A_\mu^\ast +2d^{(\nu)}J-d^{(\nu)}(N+1).
\end{align*}
Therefore we have that
\begin{multline}
\label{eq:first-term}
e^{-xA_\mu^\ast}  (\Delta^{(\nu)})^{-1}  \, A_\mu \, \Delta^{(\nu+1)} e^{xA_\mu^\ast} 
= -d^{(\nu)}x^2 A_\mu^\ast +xd^{(\nu)} (2J-N-1)
+(\Delta^{(\nu)})^{-1} \, A_\mu \, \Delta^{(\nu+1)} 
\end{multline}
Adding \eqref{eq:third-term}, \eqref{eq:second-term} and \eqref{eq:first-term} shows 
that $\Psi^{(\alpha, \nu)}$ is a polynomial of degree one.
\end{proof}

For future reference we state Corollary \ref{cor:Pearson-Equations} as
an immediate consequence of the proofs of Propositions \ref{prop:Phinu} and \ref{prop:Psinu}.

\begin{cor}
\label{cor:Pearson-Equations}
Assuming the conditions \eqref{eq:condition_delta_Phi} and \eqref{eq:recursion-alphas}, the matrix valued polynomials 
\begin{align*}
 L^{(\alpha)}_\mu(0)^\ast \Phi^{(\alpha, \nu)}(x) (L_\mu^{(\alpha)}(0)^\ast)^{-1}  &=-d^{(\nu)}x^2 A_\mu^\ast + x \left(d^{(\nu)} J + c^{(\nu)}\right),\\
 L_\mu^{(\alpha)}(0)^\ast \Psi^{(\alpha, \nu)}(x) (L_\mu^{(\alpha)}(0)^\ast)^{-1}  &= x  \left(d^{(\nu)}(J-A_\mu^\ast(J+\nu+1)-N-1)-c^{(\nu)} \right) \\
&\hspace{-0.3cm} + \left((\nu+J+1)(d^{(\nu)}J+c^{(\nu)}) + (\Delta^{(\nu)})^{-1} \, A _\mu \,  \Delta^{(\nu+1)}\right).
\end{align*}
satisfy the Pearson equations
$$\Phi^{(\alpha, \nu)}(x)=(W_\mu^{(\alpha, \nu)}(x))^{-1}W_\mu^{(\alpha, \nu+1)}(x),\qquad \Psi^{(\alpha, \nu)}(x)=(W_\mu^{(\alpha, \nu)}(x))^{-1}\frac{dW_\mu^{(\alpha, \nu+1)}}{dx}(x).$$
\end{cor}

\begin{rmk}
Upon replacing $\de^{(\nu+1)}_k=(kd^{(\nu)}+c^{(\nu)})\de^{(\nu)}_k$ in \eqref{eq:recursion-alphas}, 
we can iterate the resulting identity to obtain
\begin{equation}
\label{eq:iterated-recursion-alphas}
\frac{\delta_k^{(\nu)}}{\mu_k^2}=\frac{(1+\frac{c^{(\nu)}}{d^{(\nu)}})_{k-1}}{(k-1)!(N-k+1)_{k-1}}\, \frac{\delta^{(\nu)}_1}{\mu_1^2}.
\end{equation}
This relation can now be used to evaluate explicitly the the $0$-th moment $H_0^{(\nu, \nu)}$ given in Proposition \eqref{prop:norm-zero}. Indeed, 
\begin{align*}
(H_0^{(\nu, \nu)})_{j,j} &= \frac{\mu_j^2 \, \delta^{(\nu)}_1 \,  \Gamma(\nu+j+1)}{\mu_1^2\, (j-1)!} \,\sum_{k=1}^N \frac{(1+\frac{c^{(\nu)}}{d^{(\nu)}})_{k-1} (-j+1)_{k-1}}{(N-k+1)_{k-1}(k-1)!}\\
&=\frac{\mu_j^2 \,\delta^{(\nu)}_1 \, \Gamma(\nu+j+1)}{\mu_1^2\, (j-1)!} \, \rFs{2}{1}{1+\frac{c^{(\nu)}}{d^{(\nu)}}, -(j-1)}{-N+1}{1}
\\&
=\frac{\mu_j^2 \, \delta^{(\nu)}_1 \,\Gamma(\nu+j+1)(-N-\frac{c^{(\nu)}}{d^{(\nu)}})_{j-1}}{\mu_1^2\, (j-1)!(-N-1)_{j-1}},
\end{align*}
where the ${}_2F_1$ is summed by the Chu-Vandermonde Identity.
\end{rmk}	

\section{Shift Operators}\label{sec:shift}

In this section we use the Pearson equations to give explicit  lowering and rising operators for the polynomials $P_n^{(\alpha, \nu)}$. Next we exploit the existence of the shift operators 
to give an explicit Rodrigues formula, to calculate the squared norms as well as the coefficients
in the three-term recurrence relation. Moreover, we find another matrix valued differential operator
to which the matrix polynomials are eigenfunctions. For this explicit 
matrix valued differential operator it is possible to perform a Darboux transform, and we give an
explicit expression for the Darboux transformation. We end by obtaining a Burchnall type identity, 
see \cite{IsmaKR-MV}, for the matrix valued Laguerre polynomials, and by showing that there are
at least three families of solutions to the non-linear conditions \eqref{eq:condition_delta_Phi} and 
\eqref{eq:recursion-alphas}. In this section we assume that these conditions are satisfied, 
and hence that the Pearson equations of Corollary \ref{cor:Pearson-Equations} hold. 

For matrix valued functions $P$ and $Q$, we denote by
$$
\langle P,Q\rangle^{(\alpha, \nu)} =\int_0^\infty P(x)W^{(\alpha, \nu)}_\mu(x)Q(x)^\ast \, dx,
$$
whenever the integral converges. Moreover, we have
\begin{align*}
 &\langle \frac{dP}{dx},Q\rangle^{(\alpha, \nu+1)} = \int_0^\infty \frac{dP}{dx}(x) W^{(\al,\nu+1)}_\mu(x) \bigl( Q(x)\bigr)^\ast\, dx \\
= & -\int_{0}^\infty P(x)  W_\mu^{(\alpha, \nu)}(x)\Psi^{(\alpha, \nu)}(x) Q(x)^\ast \, dx  - 
\int_{0}^\infty P(x) W_\mu^{(\nu)}(x)\Phi^{(\alpha, \nu)}(x) 
\frac{dQ^\ast}{dx}(x)\, dx\\
= &\, - \langle P, QS^{(\alpha, \nu)}\rangle^{(\alpha, \nu)},
\end{align*}
where $S^{(\alpha, \nu)}$ is the first order matrix valued differential operator
\begin{equation}
\label{eq:operator_Snu}
(QS^{(\alpha, \nu)})(x)=\frac{dQ}{dx}(x)(\Phi^{(\alpha, \nu)}(x))^*+Q(x)(\Psi^{(\alpha, \nu)}(x))^\ast.
\end{equation}
Note that we have to assume that the decay at $0$ and at $\infty$ is sufficiently large, which is the case for 
e.g. polynomials $P$ and $Q$. 

In particular, if we set $P(x)=P_n^{(\alpha, \nu)}(x)$ and $Q(x)=x^k$, considering the degrees of $\Phi^{(\alpha, \nu)}$ and $\Psi^{(\alpha, \nu)}$ we obtain $\langle \frac{dP_n^{(\alpha, \nu)}}{dx},Q\rangle^{(\alpha, \nu+1)}=0$ for all $k\in \mathbb{N}$, $k<n$. Since the leading coefficient of $\frac{dP^{(\alpha, \nu)}_n}{dx}$ is non-singular, we conclude that $\{\frac{dP^{(\alpha, \nu)}_n}{dx}\}_n$ is a sequence of matrix valued orthogonal polynomials with respect to $W_\mu^{(\alpha, \nu+1)}$. 
Similarly, the sequence $\{P_n^{(\alpha, \nu+1)}S^{(\alpha, \nu)}\}_n$ is a sequence of matrix valued orthogonal polynomials with respect to $W_\mu^{(\alpha, \nu)}$.

\begin{prop}
\label{prop:shifts}
Assume that the conditions \eqref{eq:condition_delta_Phi} and \eqref{eq:recursion-alphas} are satisfied for all 
$\nu$ of the form $\nu_0+k$,  $k\in \mathbb{N}$, for some fixed $\nu_0$. 
The first order differential operator $S^{(\alpha, \nu)}$ given in \eqref{eq:operator_Snu} satisfies
$$
\langle \frac{dP}{dx}, Q\rangle^{(\alpha, \nu+1)} =  -\langle P, QS^{(\alpha, \nu)}\rangle^{(\alpha, \nu)},
$$
for matrix valued polynomials $P$ and $Q$. Moreover
\begin{equation*}
\label{eq:shifts}
\frac{dP^{(\alpha, \nu)}_{n}}{dx}(x)=n\, P_{n-1}^{(\alpha, \nu+1)}(x), \qquad (P^{(\alpha, \nu+1)}_{n-1}S^{(\alpha, \nu)})(x)=K^{(\alpha,\nu)}_nP_{n}^{(\alpha, \nu)}(x),
\end{equation*}
where the matrices $K^{(\alpha,\nu)}_{n}$ are invertible and are explicitly given by
\begin{equation*}
L_\mu^{(\alpha)}(0)^{-1} K^{(\alpha, \nu)}_{n}  L_\mu^{(\alpha)}(0) =d^{(\nu)}(J-(J+\nu+n)A_\mu-N-1)-c^{(\nu)} .
\end{equation*}
\end{prop}

\begin{proof}
Taking into account the preceding discussion, we have that  $\frac{dP^{(\alpha,\nu)}_{n}}{dx}$ is a multiple of  $P_{n-1}^{(\alpha, \nu+1)}$ and $P^{(\alpha, \nu+1)}_{n-1}S^{(\alpha, \nu)})$ is a multiple of $P_{n}^{(\alpha, \nu)}(x)$, these multiples follow from the leading coefficients. Now we only need to show that $K_n^{(\alpha, \nu)}$ is invertible. Observe that $L_\mu^{(\alpha)}(0)^{-1} K^{(\alpha, \nu)}_{n}  L_\mu^{(\alpha)}(0)$ is a lower triangular matrix whose $j$-th diagonal entry is given by
$$
(L_\mu^{(\alpha)}(0)^{-1} K^{(\alpha, \nu)}_{n}  L_\mu^{(\alpha)}(0))_{j,j}=d^{(\nu)}(j-(N+1))-c^{(\nu)}.
$$
These entries are strictly negative, since $c^{(\nu)}$ and $d^{(\nu)}$ are positive.
So invertibility follows.
\end{proof}

We note that the matrix $\widetilde K^{(\alpha, \nu)}_{n}=  L_\mu^{(\alpha)}(0)^{-1} K^{(\alpha, \nu)}_{n}  L_\mu^{(\alpha)}(0)$ is actually a function of $\nu+n$ so that $\widetilde K^{(\alpha, \nu+j)}_{n-j}= \widetilde K^{(\alpha, \nu)}_{n}$ for all $j\leq n$. Now conjugating with $L_\mu^{(\alpha)}(0)^{-1}$ we obtain that 
\begin{equation}
\label{eq:K-commute}
 K^{(\alpha, \nu+j)}_{n-j}=K^{(\alpha, \nu)}_{n},\qquad j\leq n.
 \end{equation}

\begin{thm}
\label{thm:rodrigues}
Assume that the conditions of Proposition \ref{prop:shifts}. The polynomials $P_{n}^{(\alpha, \nu)}$ satisfy the following Rodrigues formula
$$P^{(\alpha, \nu)}_{n}(x)=G^{(\alpha, \nu)}_{n} \left( \frac{d^nW_\mu^{(\alpha, \nu+n)}}{dx^n}(x)\right) W_\mu^{(\alpha, \nu)}(x)^{-1},$$
where 
$G^{(\alpha, \nu)}_{n}=(K_{n}^{(\alpha, \nu)})^{-1}\cdots (K_{1}^{(\alpha, \nu+n-1)})^{-1} = (K_{n}^{(\alpha, \nu)})^{-n}$. 
 Moreover, the squared norm $H_{n}^{(\alpha, \nu)}$ is given by
\begin{align*}
H_{n}^{(\alpha, \nu)}&=(-1)^n \, n! \, (K_{n}^{(\alpha, \nu)})^{-n}  H_{0}^{(\alpha, \nu+n)}\\
&= (-1)^n \, n! \, M^{(\alpha,\nu+n)}_\mu (K_{n}^{(\nu+n, \nu)})^{-n}  H_{0}^{(\nu+n, \nu+n)} (M^{(\alpha,\nu+n)}_\mu)^\ast,
\end{align*}
where $M^{(\nu+n,\nu)}_\mu$ is the matrix given in \eqref{eq:weight-alpha-lambda}.
\end{thm}

\begin{proof}
Observe that $(QS^{(\alpha, \nu)})(x) = \frac{d(QW_\mu^{(\alpha, \nu+1)})}{dx}(x) \bigl(W_\mu^{(\alpha, \nu)}(x)\bigr)^{-1}$ by Corollary \ref{cor:Pearson-Equations}. Iterating gives
\[
\bigl( Q S^{(\alpha, \nu+n-1)} \cdots S^{(\alpha, \nu)}\bigr)(x) = \frac{d^n(QW_\mu^{(\alpha, \nu+n)})}{dx^n}(x) \bigl(W_\mu^{(\alpha, \nu)}(x)\bigr)^{-1}.
\]
Now taking $Q(x)=P_{0}^{(\alpha, \nu+n)}(x)=1$ and using Proposition \ref{prop:shifts} repeatedly we obtain the Rodrigues formula.

Finally, for the squared norm we observe that
\begin{multline*}
nH^{(\alpha, \nu+1)}_{n-1}= n\langle P_{n-1}^{(\alpha, \nu+1)},P_{n-1}^{(\alpha, \nu+1)}\rangle^{(\alpha, \nu+1)} = \langle \frac{dP_{n}^{(\alpha, \nu)}}{dx},P_{n-1}^{(\alpha, \nu+1)}\rangle^{(\alpha, \nu+1)} \\
= \langle P_{n}^{(\alpha, \nu)},P_{n-1}^{(\alpha, \nu+1)}S^{(\alpha, \nu)}\rangle^{(\alpha, \nu)} =  \langle P_{n}^{(\alpha, \nu)},P_{n}^{(\alpha, \nu)}\rangle^{(\alpha, \nu)}(K_{n}^{(\alpha, \nu)})^*= H^{(\alpha, \nu)}_{n}(K_{n}^{(\alpha, \nu)})^*,
\end{multline*}
so that $H^{(\alpha, \nu)}_{n} = n \, (K_{n}^{(\alpha, \nu)})^{-1} H^{(\alpha, \nu+1)}_{n-1}$, where we have used that  $H^{(\alpha, \nu)}_{n}$ is self-adjoint for all $n\in \mathbb{N}$. Iterating and using \eqref{eq:K-commute} and  \eqref{eq:norms-diff-parameters} gives the expressions for the squared norm.
\end{proof}

\begin{cor}
\label{cor:second_DO_Laguerre}
Assume the conditions of Proposition \ref{prop:shifts}. The second-order differential operator
$$
\mathcal{D}^{(\alpha, \nu)} = S^{(\alpha, \nu)}\circ \frac{d}{dx} = 
\left( \frac{d^2}{dx^2} \right) \Phi^{(\alpha, \nu)}(x)^\ast  + \left( \frac{d}{dx} \right) \Psi^{(\alpha, \nu)}(x)^\ast,
$$
is symmetric with respect to the weight $W^{(\alpha, \nu)}_{\mu}$. Moreover, for all $n\in\N$ we have
$$
P_{n}^{(\alpha, \nu)}\mathcal{D}^{(\alpha, \nu)}=
\Lambda^{(\alpha, \nu)}_nP_{n}^{(\alpha, \nu)},\qquad \Lambda_n^{(\alpha, \nu)}=nK_{n}^{(\alpha,\nu)}.$$
Moreover, the operators $\mathcal{D}^{(\alpha, \nu)}$ and $D^{(\alpha, \nu)}$ commute.
\end{cor}

\begin{proof}
The fact that $\mathcal{D}^{(\alpha, \nu)}$ is symmetric with respect to $W^{(\alpha, \nu)}_{\mu}$ follows directly from the factorization $\mathcal{D}^{(\alpha, \nu)}=S^{(\alpha, \nu)}\circ \frac{d}{dx}$ and Proposition \ref{prop:shifts}. Then the orthogonal polynomials $P_n^{(\alpha, \nu)}$ are eigenfunctions of $\mathcal{D}^{(\alpha, \nu)}$ and the eigenvalue is obtained by looking at the leading coefficients.

In order to prove that $\mathcal{D}^{(\alpha, \nu)}$ and $D^{(\alpha, \nu)}$ commute, we will show that the corresponding eigenvalues $\Gamma^{(\alpha,\nu)}_n$ and $\Lambda^{(\alpha,\nu)}_n$ commute, see \cite[Cor.~4.4]{PR}. It is then enough to show that the following matrices commute;
$$
\widetilde \Lambda^{(\alpha,\nu)}_n=(L^{(\alpha)}_\mu(0))^{-1} \Lambda^{(\alpha,\nu)}_n L^{(\alpha)}_\mu(0) ,\qquad \widetilde \Gamma^{(\alpha,\nu)}_n = (L^{(\alpha)}_\mu(0))^{-1}  \Gamma^{(\alpha,\nu)}_n  L^{(\alpha)}_\mu(0).
$$
Using the explicit expressions of $\Gamma^{(\alpha,\nu)}_n$, Proposition \ref{lem:commutation_reltatios} and Proposition \ref{eq:shifts} we obtain
$$ 
\widetilde \Gamma^{(\alpha,\nu)}_n-(\alpha-\nu-n)=(J+\nu+n)A_\mu-J=-\frac{ \widetilde \Lambda_n^{(\alpha, \nu)} -  c^{(\nu)}}{d^{(\nu)} } -N-1.
$$
So $\Gamma^{(\alpha,\nu)}_n$ and $\Lambda_n^{(\alpha, \nu)}$ commute for all $n$. 
\end{proof}

\begin{rmk}
\label{rmk:darboux-hermite}
Corollary \ref{cor:second_DO_Laguerre} states that the differential operator $\mathcal{D}^{(\alpha, \nu)}$
has a factorization, where first the lowering operator $\frac{d}{dx}$ is applied
and next the raising operator $S^{(\al,\nu)}$. 
The Darboux transform of such a differential operator is obtained by interchanging the order of 
the lowering and raising operator. This gives a differential operator $\widetilde{\mathcal{D}}^{(\alpha, \nu)} = \frac{d}{dx}\circ S^{(\nu)}$ which has the orthogonal polynomials $P_n^{(\alpha, \nu+1)}$ as eigenfunctions. Explicitly we have
\begin{multline*}
\Bigl( P^{(\al,\nu+1)}_n\tilde{\cD}^{(\al,\nu)}\Bigr)(x) = 
\frac{d^2P^{(\al,\nu+1)}_n}{dx^2}(x) \Phi^{(\al,\nu)}(x)^\ast
+ \frac{dP^{(\al,\nu+1)}_n}{dx}(x)
\left(\frac{d\Phi^{(\al,\nu)}}{dx}(x)^\ast
+ \Psi^{(\al,\nu)}(x)^\ast \right) 
\\ +  P^{(\al,\nu+1)}_n(x) \frac{d\Psi^{(\al,\nu)}}{dx}(x)^\ast =
\Xi_n^{(\al,\nu+1)} P^{(\al,\nu+1)}_n(x), 
\end{multline*}
where the eigenvalue  $\Xi_n^{(\al,\nu+1)}$ is given by
$$ \Xi_n^{(\al,\nu+1)} =n^2 \text{lc}(\Phi^{(\al,\nu)})^\ast + (n+1)\, \text{lc}(\Psi^{(\al,\nu)})^\ast.$$
\end{rmk}

\begin{prop}\label{prop:Darbouxexpl}
With the notations as in Proposition \ref{prop:symmetry_D}, Corollary \ref{cor:second_DO_Laguerre}
and Remark \ref{rmk:darboux-hermite} and assuming 
$\frac{c^{(\nu)}}{d^{(\nu)}}= \nu +\rho$ for some constant $\rho$ we have 
\begin{gather*}
\frac{1}{d^{(\nu)}}P\tilde{\cD}^{(\al,\nu)} = \frac{1}{d^{(\nu+1)}}  \, P{\cD}^{(\al,\nu+1)} 
- P D^{(\al,\nu+1)}  +  \al -N-2\nu-2-\rho.
\end{gather*}
\end{prop}

Note that the assumption on the quotient is satisfied in Examples \ref{ex:ax1}, \ref{ex:ax2} and \ref{ex:ax3}.

\begin{proof} The proof is a bit involved, but essentially straightforward. 
First, to use the explicit expression of Remark \ref{rmk:darboux-hermite} we
need the explicit expressions for $\Phi^{(\al,\nu)}(x)^\ast$, $\Psi^{(\al,\nu)}(x)^\ast$
and compare the difference for $\nu$ and $\nu+1$. 
From Corollary \ref{cor:Pearson-Equations} we have 
\begin{gather*}
\frac{1}{d^{(\nu)}} \Phi^{(\al,\nu)}(x)^\ast -
\frac{1}{d^{(\nu+1)}} \Phi^{(\al,\nu+1)}(x)^\ast = 
x\left( \frac{c^{(\nu)}}{d^{(\nu)}}- \frac{c^{(\nu+1)}}{d^{(\nu+1)}}\right)I = -xI, \\
\frac{1}{d^{(\nu)}} \left(\frac{d\Phi^{(\al,\nu)}}{dx}(x)\right)^\ast
= L_\mu^{(\al)}(0)\Bigl( -2x A_\mu +J + \nu +\rho \Bigr) L_\mu^{(\al)}(0)^{-1} \\
= -2x\bigl( 1-(1+A_\mu)^{-1}\bigr) +(\al+J)A_\mu+J+\nu +\rho
\end{gather*}
using Lemma \ref{lem:commutation_reltatios}. To do the same for $\Psi^{(\al,\nu)}(x)^\ast$
we first observe that $\frac{1}{d^{(\nu)}} (\De^{(\nu)})^{-1} A_\mu \De^{(\nu+1)}$ 
is actually independent of $\nu$, because of \eqref{eq:recursion-alphas}. 
So we find from Corollary \ref{cor:Pearson-Equations} and Lemma \ref{lem:commutation_reltatios} 
\begin{gather*}
\frac{1}{d^{(\nu)}} \Psi^{(\al,\nu)}(x)^\ast -
\frac{1}{d^{(\nu+1)}} \Psi^{(\al,\nu+1)}(x)^\ast 
= L_\mu^{(\al)}(0)\Bigl( x(1+A_\mu) -2J -2-2\nu-\rho \Bigr) L_\mu^{(\al)}(0)^{-1} \\
= x+x(1-(1+A_\mu)^{-1}) -2J -2(\al+J)A_\mu -2-2\nu-\rho.
\end{gather*}
Finally, by Corollary \ref{cor:Pearson-Equations} and Lemma \ref{lem:commutation_reltatios} 
\begin{gather*}
\frac{1}{d^{(\nu)}} \left(\frac{d\Psi^{(\al,\nu)}}{dx}(x)\right)^\ast
= L_\mu^{(\al)}(0)\Bigl( J - (J+\nu+1)A_\mu -N-1 -\frac{c^{(\nu)}}{d^{(\nu)}} \Bigr) L_\mu^{(\al)}(0)^{-1} \\
= (\al+J)A_\mu+J - ((\al+J)A_\mu+J +\nu+1)(1-(1+A_\mu)^{-1}) -N-1-\nu-\rho \\
= ((\al+J)A_\mu+J +\nu+1)(1+A_\mu)^{-1} -N-2\nu-2-\rho\\ 
= J + (\al A_\mu +\nu+1)(1+A_\mu)^{-1}  -N-2\nu-2-\rho. 
\end{gather*}
Collecting these expressions in the differential operator in Remark \ref{rmk:darboux-hermite} gives
\begin{gather*}
\frac{1}{d^{(\nu)}}P\tilde{\cD}^{(\al,\nu)}(x) -  \frac{1}{d^{(\nu+1)}}  \, P{\cD}^{(\al,\nu+1)}(x)  
= \\ \frac{d^2P}{dx^2}(x) (-xI) +  
\frac{dP}{dx}(x) 
\left( x(1+A_\mu)^{-1}-J -(\al+J)A-2-\nu  \right) 
+ \\
P(x) \left( J + (\al A_\mu+\nu+1)(1+A_\mu)^{-1}  -N-2\nu-2-\rho \right) 
= \\
- (PD^{(\al,\nu+1)})(x) +  \al -N-2\nu-2-\rho. \qedhere
\end{gather*}
\end{proof}

As a next application of the shift operators, we calculate the matrix coefficients in the 
three-term recurrence for the monic polynomials explicitly. 
As in \eqref{eq:three-term-r}, the monic matrix valued Laguerre-type polynomials $P^{(\alpha,\nu)}_n$ satisfy a three-term recurrence relation of the form
\begin{equation*}
xP^{(\alpha,\nu)}_{n}(x)=P^{(\alpha,\nu)}(x)+B^{(\alpha, \nu)}_{n}P^{(\alpha,\nu)}_n(x)+C_{n}^{(\alpha, \nu)}P^{(\alpha,\nu)}_{n-1}(x).
\end{equation*}
Proceeding as in \cite[\S 5.3]{Koe:Rio:Rom}, the coefficients $B^{(\alpha, \nu)}_{n}$ and $C^{(\alpha, \nu)}_{n}$, are given by
\begin{equation}
\label{eq:rec-coeff-B}
B_{n}^{(\alpha, \nu)} = X_{n}^{(\alpha, \nu)}-X^{(\alpha, \nu)}_{n+1},\qquad C_{n}^{(\alpha, \nu)} = H_{n}^{(\alpha, \nu)}(H_{n-1}^{(\alpha, \nu)})^{-1},
\end{equation}
where  $X_{n}^{(\alpha, \nu)}$ is the one-but-leading coefficient of $P_{n}^{(\alpha, \nu)}$.

\begin{prop}\label{prop:3termrecur}
Assume the conditions of Proposition \ref{prop:shifts}.  The coefficients of the three-term recurrence relation 
for the monic Laguerre-type orthogonal polynomials are given by
\begin{align*}
B_{n}^{(\al,\nu)}&= n\,(K_1^{(\alpha, \nu+n-1)} )^{-1}(\Psi^{(\alpha, \nu+n-1)}(0))^* - (n+1)\,(K_1^{(\alpha, \nu+n)} )^{-1}(\Psi^{(\alpha, \nu+n)}(0))^*, \\
C_{n}^{(\al,\nu)}&= -n M^{(\alpha,\nu+n)}_\mu (K_n^{(\nu+n,\nu)})^{-n} H_0^{(\nu+n,\nu+n)} (1-A^\ast_\mu) (H_0^{(\nu+n-1,\nu+n-1)})^{-1} \\
&\hspace{7cm} \times (K_{n-1}^{(\nu+n-1,\nu)})^{n-1} (M^{(\alpha,\nu+n-1)}_\mu)^{-1}
\end{align*}
\end{prop}

\begin{proof}
By taking the derivative of $P_{n}^{(\alpha, \nu)}$ with respect to $x$ and using Proposition \ref{prop:shifts}, we find that
$(n-1)X_{n}^{(\alpha, \nu)}=nX^{(\alpha, \nu)}_{n+1}$ which gives $X_{n}^{(\alpha, \nu)}=nX_{1}^{(\alpha, \nu+n-1)}$. Using the Rodrigues formula we obtain $P_{1}^{(\alpha, \nu)}(x)=G_{1}^{(\alpha, \nu)}(\Psi^{(\alpha, \nu)}(x))^*$. Evaluating at $x=0$ gives
$$X_{n}^{(\alpha, \nu)} =  n\,(K_1^{(\alpha, \nu+n-1)} )^{-1}(\Psi^{(\alpha, \nu+n-1)}(0))^*.$$
Replacing $X_{n}^{(\alpha, \nu)} $ in \eqref{eq:rec-coeff-B} we obtain the expression for $B_{n}^{(\alpha, \nu)}$.

On the other hand, using the expression for the norm in Theorem \ref{thm:rodrigues}, we find that 
\begin{multline*}
H_{n}^{(\alpha, \nu)}(H_{n-1}^{(\alpha, \nu)})^{-1}=-nM^{(\alpha,\nu+n)}_\mu (K_n^{(\nu+n,\nu)})^{-n} H_0^{(\nu+n,\nu+n)} (M^{(\alpha,\nu+n)}_\mu)^\ast((M^{(\alpha,\nu+n-1)}_\mu)^\ast)^{-1} \\
\times  (H_0^{(\nu+n-1,\nu+n-1)})^{-1} (K_{n-1}^{(\nu+n-1,\nu)})^{n-1} (M^{(\alpha,\nu+n-1)}_\mu)^{-1}.
\end{multline*}
Now we use Lemma \ref{lem:relation-alpha-lambda} to write $L_\mu^{(\nu+n)}(0)^\ast (L_\mu^{(\nu+n-1)}(0))^\ast)^{-1}  = (1-A_\mu^\ast)$ and this completes the proof of the proposition.
\end{proof}

As a final application of the shift operators we discuss briefly an expansion formula for the 
matrix valued Laguerre polynomials arising from the Burchnall formula for matrix valued polynomials
satisfying a Rodrigues formula. Note that the matrix valued Laguerre polynomials form an example
of the general conditions in \cite[\S 4]{IsmaKR-MV}. In particular, Burchnall's formula 
\cite[Thm.~4.1]{IsmaKR-MV} applies. Assuming $\frac{c^{(\nu)}}{d^{(\nu)}}=\nu + \rho$ as in 
Proposition \ref{prop:Darbouxexpl}, we see from Corollary \ref{cor:Pearson-Equations} that
\begin{gather*}
\left( \Phi^{(\al,\nu)}(x)\cdots  \Phi^{(\al,\nu+k-1)}(x) \right)^\ast =
(-1)^k x^k \left( \prod_{p=0}^{k-1} d^{(\nu+p)}\right) 
(J-xA_\mu+\nu+\rho)_k.
\end{gather*}
Moreover, since the matrices $G^{(\al,\nu)}_n$ in 
Theorem \ref{thm:rodrigues} are powers, the result of \cite[Cor~4.2]{IsmaKR-MV} 
simplifies and we obtain Corollary \ref{cor:Burchnallexpansion}.

\begin{cor}\label{cor:Burchnallexpansion} Under the conditions of Proposition \ref{prop:Darbouxexpl}
we have the following expansion formula for the matrix valued Laguerre polynomials. 
\begin{gather*}
(K^{(\al,\nu)}_{n+m})^{-n} P^{(\al,\nu)}_{n+m}(x) = 
\sum_{k=0}^{m\wedge n} (-1)^k \binom{n}{k} \binom{m}{k} k! 
\Bigl( \prod_{p=0}^{k-1} d^{(\nu+p)}\Bigr) x^k \\ 
\times P^{(\al,\nu+n-k)}_{m-k}(x) (K^{(\al,\nu)}_n)^{n-k}
P^{(\al,\nu+k)}_{n-k}(x) L^{(\al)}_\mu(0) (J-xA_\mu+\nu+\rho)_k  L^{(\al)}_\mu(0)^{-1}
\end{gather*}
\end{cor}

\begin{rmk}\label{rmk:Toda} In the scalar case the Burchnall identities for some subclasses 
of the Askey scheme, notably Hermite, Laguerre, Meixner-Pollaczek, Krawtchouk, Meixner and
Charlier polynomials, 
can be used to find expressions for the orthogonal polynomials for the corresponding 
Toda modification of the weight, i.e. multiplication by $e^{-xt}$, see \cite[Prop.~7.1]{IKR}. 
These are precisely the cases where it is easy to ``glue'' on the exponent $e^{-xt}$ 
to the classical weight function. 
In the matrix case, the Toda modification of the matrix weight 
leads to solutions of the  non-abelian Toda lattice, see \cite[\S 2.1]{IsmaKR-MV} 
and references given there, and this is worked out in detail for the matrix valued Hermite
polynomials in \cite[\S 5]{IsmaKR-MV}. In the case of the matrix valued Laguerre polynomials, however,
we are not lead to a corresponding solution of the non-abelian Toda lattice. 
Essentially, the Burchnall 
approach fails due to the fact that $\Phi^{(\al,\nu)}\cdots  \Phi^{(\al,\nu+k-1)}$ is
a polynomial of degree $2k$ (instead of $k$), see the discussion in \cite{IKR}. 
Even though in the matrix valued case it is straightforward to glue on the exponential
$e^{-xt}$ to the matrix weight $W^{(\al,\nu)}_\mu(x)$, it leads to a linear change in the parameters 
$\mu$ and $\de^{(\nu)}_k$, and since, the conditions \eqref{eq:condition_delta_Phi} and
\eqref{eq:recursion-alphas} form non-linear conditions, the Pearson equations do not hold
for the Toda modification. So in particular, Proposition \ref{prop:3termrecur} is no
longer valid for the Toda modified matrix weight. 
\end{rmk}

\subsection{Examples}

We cannot give all the solutions to the \eqref{eq:condition_delta_Phi} and \eqref{eq:recursion-alphas} in general, due to the non-linearity of these relations. However, if we choose coefficients $\mu_k$ such that the quotient $\mu_{k+1}^2/\mu_k^2$ coincides with some of the factors in the product $k(N-k)$ of the right hand side of \eqref{eq:recursion-alphas}, we can give explicit examples of Laguerre-type matrix valued orthogonal polynomials. In the case of Hermite-type matrix valued orthogonal polynomials given in \cite{IsmaKR-MV} there are two nonlinear conditions. The first one is the same as \eqref{eq:condition_delta_Phi} and the second one differs from \eqref{eq:recursion-alphas} by a factor of $\frac12$ on the right-hand side. The examples in this paper are obtained in the same way as in \cite{IsmaKR-MV}.

\begin{example}\label{ex:ax1}
If we assume that $\mu_{k+1}=\sqrt{N-k}\, \mu_k$, then the left hand side of \eqref{eq:recursion-alphas} coincides with the factor $N-k$ on the right hand side. Thus $\mu_k=\sqrt{(N-k+1)_k}$ for $k=1,\ldots,N$. Then \eqref{eq:condition_delta_Phi} and \eqref{eq:recursion-alphas} give the following recurrence relations
\begin{equation}
\label{eq:rec_delta_example1}
\delta_k^{(\nu+1)}=(d^{(\nu)}k+c^{(\nu)})\delta_k^{(\nu)},\qquad \delta_{k+1}^{(\nu)} = \left(k+\frac{c^{(\nu)}}{d^{(\nu)}}\right) \delta_{k}^{(\nu)}.
\end{equation}
One solution to \eqref{eq:rec_delta_example2} is given by
$$\delta_k^{(\nu)} =\Gamma(\nu)(\nu)_k,\qquad c^{(\nu)} = \nu, \qquad d^{(\nu)}=1.$$ 
\end{example}

\begin{example}\label{ex:ax2}
If we assume that $\mu_{k+1}=\sqrt{k(N-k)}\, \mu_k$, then the left hand side of \eqref{eq:recursion-alphas} coincides with the factor $k(N-k)$ on the right hand side.
Then $\mu_k=\sqrt{(k-1)! (N-k+1)_{k-1}}$, $k=1,\ldots,N$. The relations \eqref{eq:condition_delta_Phi} and \eqref{eq:recursion-alphas} give the following recurrence relations
\begin{equation}
\label{eq:rec_delta_example2}
\delta_k^{(\nu+1)}=(d^{(\nu)}k+c^{(\nu)})\delta_k^{(\nu)},\qquad \delta_{k+1}^{(\nu)} = \left(k+\frac{c^{(\nu)}}{d^{(\nu)}}\right) \delta_{k}^{(\nu)}.
\end{equation}
A solution to \eqref{eq:rec_delta_example2} is given by 
$$\delta_k^{(\nu)} =\lambda^\nu\, \Gamma(\nu+k)=(\nu)_k\lambda^\nu\Gamma(\nu),\qquad c^{(\nu)} = \nu\lambda, \qquad d^{(\nu)}=\lambda.$$ 
for some fixed $\lambda>0$.
\end{example}

\begin{example}\label{ex:ax3}
Now we take $\mu_k=1$ for all $k=1,\ldots,N$. Therefore the relations \eqref{eq:condition_delta_Phi} and \eqref{eq:recursion-alphas} are given by
\begin{equation*}
1=\frac{d^{(\nu)}k(N-k)}{(d^{(\nu)}k+c^{(\nu)})} \frac{\de^{(\nu)}_{k+1}}{\de^{(\nu)}_{k}}, \qquad \de^{(\nu+1)}_{k} = (d^{(\nu)}k+c^{(\nu)}) \de^{(\nu)}_{k+1}
\end{equation*}
for which 
\begin{equation*}
d^{(\nu)}=\rho, \qquad c^{(\nu)} = C+ \nu\rho, 
\qquad \de^{(\nu)}_{k} = \frac{(1+\nu + C/\rho)_{k-1}}{(k-1)!\, (N-k+1)_{k-1}} \rho^\nu \,\Ga(1+\nu +C/\rho)
\end{equation*}
with $\rho>0$, $\nu>0$ and $C\geq 0$ gives a solution meeting all the conditions.  

\end{example}



\end{document}